\documentclass[a4paper,reqno]{amsart}

\usepackage{amsmath}
\usepackage{paralist}
\usepackage{graphics}
\usepackage{epsfig}
\usepackage{amsfonts}
\usepackage{amssymb}
\usepackage{hyperref,cite}

\usepackage{amsthm}
\usepackage{enumerate}
\usepackage[mathscr]{eucal}
\usepackage{eqlist}

\def\teich{{Teichm\"{u}ller}}
\def\mobi{{M{\"{o}}bius}}

\newtheorem{theorem}{Theorem}[section]
\newtheorem{lemma}[theorem]{Lemma}

\def\ifl{\iffalse }

\numberwithin{equation}{section}

\numberwithin{equation}{section}

\theoremstyle{remark}

\begin{document}
\title[ On Ruelle's property]
{On Ruelle's property}

\author{Huo Shengjin}
\address{Department of Mathematics, Tianjin Polytechnic University, Tianjin 300387, China} \email{huoshengjin@tjpu.edu.cn}
\author{Michel Zinsmeister}
\address{Universite D'Orleans, MAPMO, Orleans Cedex 2, France} \email{ zins@unvi-orleans.fr}


\thanks{This work was supported by the Science and Technology Development Fund of Tianjin Commission for Higher Education(Grant No.2017KJ095).}
%
\begin{abstract}
In this paper we  investigate the range of validity of Ruelle's property. First, we show that every finitely-generated Fuchsian group has Ruelle's property. We also prove the existence of an infinitely-generated Fuchsian group satisfying Ruelle's property.  Concerning the negative results we first generalize Astala-Zinsmeister's results by proving that all convergence Fuchsian groups of the first kind fail to have Ruelle's property. At last, we also give some results about the second kind Fuchsian groups.
\end{abstract}

\maketitle

\section{1 Introduction }

 A Fuchsian group is a discrete {\mobi} group $G$ acting on the unit disk $\Delta$. The limit set of $G$, denoted by $\Lambda(G)$, is the set of accumulation points of the $G$-orbit of any point $z\in\Delta$. Since the action of $G$ is properly discontinuous, $\Lambda(G)\subset\partial\Delta.$ A Fuchsian  group $G$ is said to be of the first kind if  the limit set $\Lambda(G)$ is the entire circle. Otherwise, it is of the second kind. Points of the limit set $\Lambda$ naturally correspond to geodesic rays with fixed base point $z_{0}\in \Delta.$ The limit set can be written as the disjoint union of two special subsets; the conical limit set $\Lambda_{c}(G)$, which corresponds to geodesics which return to some compact set infinitely often (the recurrent geodesics) and the escaping limit set, $\Lambda_{e}(G),$ which corresponds to geodesics escaping to infinity.

 The critical exponent (or Poincar\'e exponent) of a Fuchsian group $G$ is defined as
\begin{align}\delta(G)&=\inf\{t: \sum_{g\in G}\exp(-t\rho(0,g(0)))<\infty\}\\
&=\inf\{t:\,\sum_{g\in G}(1-\vert g(0)\vert)^t<+\infty\},
\end{align}
where $\rho$ denotes the hyperbolic metric.

It has been proven in\cite {BJ}  that for any non-elementary group $G$, $\delta(G)=HD(\Lambda_c(G))$,  the Hausdorff dimension of the conical limit set.

  A Fuchsian group is said to be cocompact if the Riemann surface $\Delta/G$ is compact and cofinite if the quotient has finite hypebolic area. A Fuchsian  group $G$ is said to be of divergence type if $\Sigma_{g\in G}(1-|g(0)|)=\infty$. Otherwise, we say it is of convergence type.
  It is well known that
  $$\mathrm{cocompact}\, \subset\,\mathrm{cofinite}\,\subset\,\text{divergence-type}\,\subset\,\mathrm{first~~  kind}.$$
  All the second kind groups are of convergence type but the converse is not true, as we shall see later.\\

  We will call a Fuchsian group exceptional if it is the covering group of the sphere minus $m$ disks  and $n$ points where $1\leq m+n\leq 3, (m,n)\neq (1,0).$\\

Let $G$ be a Fuchsian group and  $\mu$ a bounded measurable function on $\Delta$ such that $||\mu||_{\infty}<1$ and
$$\mu(z)=\mu(g(z))\overline{g'(z)}/g'(z),\,z\in\Delta,\,g\in G.$$
We say that $\mu$ is a  $G$-compatible Beltrami coefficient (or complex dilatation). For a $G$-compatible Beltrami coefficient $\mu$, there is a corresponding quasiconformal mapping $f_{\mu}$ which is analytic outside $\Delta$ and such that $${\mu(z)=\frac{\partial f_{\mu}}{\partial \bar{z}}}/{\frac{\partial f_{\mu}}{\partial z}}, ~~\, a.e.~ z \in\Delta.$$ This map $f_{\mu}$ conjugates $G$ to a quasi-Fuchsian group $G_{\mu}=f_{\mu}\circ G\circ f_{\mu}^{-1}.$ We say that $G_{\mu}$ is a quasiconformal deformation of $G$.

We can generalize to quasi-Fuchsian groups the notion of conical and escaping sets: we can also define the Poincar\'e exponent of such a group by replacing $(1-\vert g(0)\vert)$ in (1.2) by $dist(g(0),\partial{f_\mu(\Delta)})$ and \cite{B, BJ} remains true in this case.

A Fuchsian group $G$ has Bowen's property if the limit set of any quasiconformal deformation of $G$ is either a circle or has Hausdorff dimension  $>1$. In 1979, R. Bowen \cite{Bo1} proved that if $G$ is a cocompact Fuchsian group, then this dichotomy property holds. Soon D. Sullivan \cite{Su2,Su3} extended Bowen's property to all cofinite groups. In 1990, K. Astala and the second author\cite{AZ0} showed that Bowen's property fails for all convergence groups of the first kind. At last, in 2001, C.J. Bishop showed that for all divergence groups, Bowen's property holds.

 We will say a Fuchsian group $G$ has Ruelle's property if for any  family of $G$-compatible Beltrami coefficients $(\mu_{t})$ which is analytic in $t\in\Delta$, the map $t\mapsto HD(\Lambda(G_{\mu_{t}}))$ is  real-analytic in $\Delta$.
In 1982, Ruelle \cite{Ru1} showed that all cocompact groups have this property. In 1997, J.W. Anderson and A.C. Rocha \cite{AR} extended this result to finitely-generated Fuchsian groups without parabolic elements.  In \cite{AZ1, AZ2}, K. Astala and the second author showed that for Fuchsian groups corresponding to Denjoy-Carleson domains or infinite $d$-dimensional "jungle gym" with $d\geq 3$ , Ruelle's property fails. In \cite{Bi3}, C.J. Bishop gave a criterion  for the failure of the Ruelle's property which applies to many divergence type examples including the $d$-dimensional "jungle gym" with $d=1,2$,  thus implying that Ruelle property is not equivalent to Bowen's one.

In this paper we continue to investigate the range of validity of Ruelle property. Firstly, by investigating the role of parabolic points and using Mauldin-Urbanski \cite{MU} techniques, we prove:
\begin{theorem}\label{main1}
 Every finitely-generated  Fuchsian group has Ruelle's property.
\end{theorem}
Using the same kind of techniques we also prove the existence of an infinitely-generated Fuchsian group with Ruelle's property:
\begin{theorem}\label{main3}
 There exists a sequence $(s_n)$ of real numbers increasing to infinity such that the Fuchsian group uniformizing $S=\mathbb{C}\backslash\{s_n,\,n\geq 0\}$ has Ruelle's property.
\end{theorem}

{\noindent {\bf Remark:} } Theorem \ref{main3} does not hold for any sequence $(s_{n}).$ For example, $\mathbb{C}\setminus\mathbb{Z}$ is a $\mathbb{Z}$-covering of the twice-punctured sphere as was noticed in \cite{AD}(we thank Mariusz Urbanski for having pointed out to us this reference), which implies by the result of Bishop \cite{Bi3} that Ruelle's property fails in this case.

Concerning the negative results we first generalize Astala-Zinsmeister's results in \cite{AZ0,AZ} by proving:
\begin{theorem}{\label{main}}
All convergence type Fuchsian groups of the first kind fail to have Ruelle's property.
  \end{theorem}
Concerning  the second kind (thus convergence type) Fuchsian groups, we prove:
\begin{theorem}\label{main2}
Let $S$ be an infinite area hyperbolic Riemann surface and $G$ be the universal covering group of $S.$ Let $\gamma$ be a closed geodesic in the surface $S.$ Cutting $S$ along $\gamma$, one obtains one or two bordered Riemann surfaces. We construct a new surface $S'$ by gluing the one of infinite area with one or two funnels along $\gamma$. If $G$ is of the first kind, then the corresponding second-kind covering group $G'$ of $S'$ fails to have Ruelle's property.
\end{theorem}

\section{Proof of Theorem \ref{main} }

In order to prove this theorem we will first need the  following lemma  from (\cite{Bi4}, Lemma 2.1).
\begin{lemma}{\label{le1}}
Suppose that $G$ is a Fuchsian group and $\mu$ is a $G$-compatible complex dilatation.
If $\{\mu_{n}\}$ is a family of $G$-compatible complex dilatations with $L^{\infty}$ norms uniformly bounded by $k<1$ and which converges pointwise to $\mu$, then $$\liminf_{n\rightarrow \infty}\delta(G_{\mu_{n}})\geq \delta(G_{\mu}). $$
\end{lemma}

For quasiconformal deformations of Fuchsian groups, C. J. Bishop\cite{Bi3} gave the following result.

\begin{lemma}{\label{le2}}
If $G$ is a torsion free non-exceptional type Fuchsian group, then $G$ has a quasiconformal deformation $G_{\mu}$ with $HD(\Lambda(G_{\mu}))\geq\delta(G_{\mu})>1.$
\end{lemma}

Before we continue the proof let us recall some facts from $BMO$-Teichm$\ddot{\mathrm u}$ller theory. A Carleson measure on the unit disk $\Delta$ is a positive measure $\nu$ such that there exists a constant $C$ such that for any $z\in\partial{\Delta}$ and any $r<1$,
$$\nu(\Delta\cap D(z,r))\leq Cr,$$
where $D(z,r)$ denotes the disk of center $z$ and radius $r$. If $G$ is a convergence-type Fuchsian group then
$$\sum (1-\vert g(0)\vert)\delta_{g(0)}$$
is a Carleson measure ($\delta_z$ stands for the Dirac mass at $z$). It follows that all convergence type Fuchsian groups $G$ have $G-$compatible Beltrami coefficients $\mu$ such that
$$\frac{\vert \mu (z)\vert ^2}{1-\vert z\vert}dxdy$$
is a Carleson measure. For these coefficients it follows that $\log (f_\mu')$ belongs to the space $BMOA(\Delta)$ with a norm controlled by the above Carleson measure norm. In particular, when the Carleson norm is small then $\partial{f_\mu(\Delta)}$ is a rectifiable (chord-arc) curve. This is essential for the proof that convergence-type first-kind Fuchsian groups fail to have Bowen property.

These properties imply the
\begin{lemma}{\label{le3}}
Suppose $G$ is a convergence type Fuchsian group and  $\mu$  is a $G$-compatible complex dilatation. If $\mu$ is compactly supported on the surface $\Delta/G$,(we say $\mu$ induces a compact deformation) then
$$\displaystyle\frac{|\mu(z)|^{2}}{1-|z|^{2}}\in CM(\Delta),$$
where $CM(\Delta)$ denotes the set of all Carleson measures of $\Delta.$
\end{lemma}
We can now prove the theorem. First, by Bishop's result Lemma \ref{le2} there exists a $G$-compatible Beltrami coefficient $\mu$ such that $\delta(G_\mu)>1$. Using Lemma\ref{le1}we may assume that $\mu$ is compactly supported: but then by Lemma\ref{le3}, $$\displaystyle\frac{|\mu(z)|^{2}}{1-|z|^{2}}\in CM(\Delta)$$ and if we consider the family $(t\mu)$ we see that $HD(\Lambda(G_{\mu})>1$ while $HD(\Lambda(G_{t\mu})=1$ for $t$ small, thus contradicting Ruelle's property.

\section{Proof of Theorem \ref{main2} }

We begin with the

{\bf Claim:} $HD(\bigwedge_{e}(G'))=1.$

  \noindent{Proof:} Recall that $S$ is recurrent (resp. transient) if the Brownian motion on $S$ is recurrent (resp. transient). The universal covering group of a recurrent (resp. transient) surface is of divergence (resp. convergence) type. To prove the claim, we need the following lemma, which is due to J.L. Fernandez and M. Melian, see (\cite{FM}, Theorem 1).

 \begin{lemma}\label{le3.1}
 Suppose $G$ is a first kind Fuchsian group such that the quotient $\Delta/G$ has infinite area. Then there are two possibilities:

  (i) If $G$ is of convergence type, then $\Lambda_{e}$ has full measure.

 (ii) If $G$ is of divergence type, then $\Lambda_{e}$ has measure zero, but its Hausdorff dimension is equal to $1.$
 \end{lemma}

 For the case  $S$ being a transient hyperbolic Riemann surface (i.e. case (i)), the proof of the claim is simple. Since in this case there exists  $z\in S$ such that the set of geodesics from $z$ going to $\infty$ without hitting $\gamma$ has positive measure, but less than $1$. It follows that  the escaping limit set of $S'$ has positive measure and the claim follows.

Suppose now that $S$ is a recurrent hyperbolic Riemann surface with infinite area.
A domain $D \subset S$ is called a geodesic domain if its relative boundary consists of finitely many non-intersecting closed simple geodesics and its area is finite. Fix a point $p\in S$, by Theorem 4.1 in \cite{FM} we know that there exists a family $\{D_{i}\}_{i=0}^{\infty}$ of pairwise disjoint geodesic domains in $S$ satisfying:

(i) The boundary of $D_{i}$ and $D_{i+1}$ have at least a simple closed geodesic in common.

(ii) $\lim_{i\rightarrow\infty}dist(p, D_{i})=\infty.$

Let $D_{i}'\subset S'$ be the isometric image of $D_{i}.$ Without loss of generality we may suppose that $\gamma$ as stated in the theorem is part of the boundary of $D_{0}.$
 For the family  $\{D_{i}\}_{i=0}^{\infty}$,   the method used to prove Lemma 1 by J.L. Fernandez and M. Melian \cite{FM} is still valid. Modeled upon their method, we  get that $HD(\Lambda_{e}(G'))=1.$ For the readers' convenience, we include some details taken from \cite{FM}.

 Let $\{D'_{i}\}_{i=0}^{+\infty}$ be the family of geodesic domains of $S'$ constructed as above.
 For any $i$, let $S'_{i}$ be the Riemann surface obtained from $D'_{i}$ by pasting a funnel along each one of the simple closed geodesics of its boundary. For each $i$, we choose a simple closed geodesic  $\gamma_{i}$ from the common boundary $D'_{i}\cap D'_{i+1}$ and a point $P_{i}\in \gamma_{i}.$ By (\cite{FM}, Theorem 4.1) and noticing that $D'_{i}$ is the isometric image of $D_{i}$, we have
 $\delta_{i}\rightarrow 1$ when $i$ tends to infinity, where $\delta_{i}$ is the Poincare exponent of $S'_{i}.$

 For $\theta\in (0, \frac{1}{2}\pi)$, by (\cite{FM}, Theorem 5.1), we can choose a collection $\mathfrak{B}_{i}$ of geodesics in $S'_{i}$ with initial and final endpoint $P_{i}$ such that
$$L_{i}\leq\text{length}(\gamma)\leq L_{i}+C(P_{i}),\, \gamma\in \mathfrak{B}_{i}.$$
 The number of geodesic arcs in $\mathfrak{B_{i}}$ is at least $e^{\sigma_{i}}$, and both the absolute value of the angles between $\gamma$ and the closed geodesic $\gamma_{i}$ are less than or equal to $\theta,$
 where $L_{i}$ is a constant such that $L_{i}\rightarrow\infty$ as $i\rightarrow\infty$, $C(q_{i})$ is a constant depending only on the length of the geodesic $G_{k_{i}}$, and $\sigma_{i}<\delta(S'_{i})$, $\sigma_{i}\rightarrow 1$ as $i\rightarrow\infty$.
 Note that for each $i$, $D'_{i}$ is the convex core of $S'_{i},$
implying that every geodesic arc $\gamma\in \mathfrak{B}_{i}$ is contained in the convex core $D'_{i}.$

  Furthermore, for each $i$, we may choose a geodesic arcs $\gamma_{i}^{*}$ with initial point $P_{i}$ and final endpoint $P_{i+1}$ such that
 $$L_{i}\leq\text{length}(\gamma)\leq L_{i}+C(P_{i+1}),$$ and both the absolute value of the angles between $\gamma_{i}$, $\gamma_{i}^{*}$, and $\gamma_{i}^{*}$, $\gamma_{i+1}$  are less than or equal to $\theta.$

 Now we are going to construct a tree $\mathfrak{T}$ consisting of oriented geodesic arcs in the unit disk $\Delta.$

 First, lift $\gamma_{0}^{*}$ to the unit disk starting at $0$ (without loss of generality we may suppose that 0 projects onto $P_{0}$).  From the endpoint of the lifted $\gamma_{0}$ (which project onto $P_{1}$), lift the family $\mathfrak{B}_{1}$; from each of the end points of these liftings (which still project onto $P_{1}$), lift again $\mathfrak{B}_{1}$. Keep lifting $\mathfrak{B}_{1}$ in this way $M_{1}$ times.

 Next, from each one of the endpoints obtained in the process above, we lift $\gamma_{1}^{*}$, and from each one of the endpoints of the lifting of $\gamma_{1}^{*}$ (which project onto $P_{2}$), we lift the collection $\mathfrak{B}_{1}$ sucessively $M_{2}$ times as above. Continuously this process
 indefinitely we obtain a tree $\mathfrak{T}.$

 It is easy to see that $\mathfrak{T}$ contains uncountably many branches.  The tips of the branches of $\mathfrak{T}$ are contained in the escaping limit set $\Lambda_{e}$ of the covering group of $S'$.

By the proof of (\cite{FM}, Theorem 1.1), we know that for suitable sequence $\{M_{i}\}$ of repetitions, the dimension of the set of the tips of the branches of $\mathfrak{T}$ is 1. By the construction of the tree $\mathfrak{T}$ we see that the tree $\mathfrak{T}$ is a unilaterally connected graph. Hence the geodesic corresponding to any branch of $\mathfrak{T}$ does not tend to the funnel with boundary $\gamma.$ Hence the dimension of the escaping limit set $\Lambda_{e}$ of the covering group $G'$ is 1.

We can now prove the theorem.

As in the proof of Theorem {\ref{main}}, by Lemma \ref{le1} and Lemma \ref{le2}, we can choose a  compactly supported $G'$-compatible Beltrami coefficient $\mu$ such that $\delta(G'_\mu)>1$. Bishop \cite{Bi2} showed that the Hausdorff dimension of the escaping limit set is unchanged under any compact deformation. Hence for the deformation group $G'_{\mu}$, we have $HD(\Lambda_{e})(G'_{\mu})=HD(\Lambda_{e}(G'))$. By Lemma \ref{le3}, $$\displaystyle\frac{|\mu(z)|^{2}}{1-|z|^{2}}\in CM(\Delta)$$ and if we also consider the family $(t\mu)$, we see that $HD(\Lambda(G_{\mu}))>1$ while $HD(f_{t\mu}(\partial\Delta))=1$ for $t$ small. However, $HD(\Lambda_{e}(G'_{\mu}))=1$ for any $t\in[0,1],$  hence $HD(\Lambda(G_{t\mu}))=1$ for $t$ small, thus contradicting Ruelle's property.\qed

\section{ Proof of Theorem \ref{main1}}
Before giving the proof of this theorem, we first recall some preliminaries.

Suppose that $G$ is a finitely generated Fuchsian group of the first kind with a set of generators containing $n$ parabolic elements. By the work of R. Bowen and C. Series \cite{BS}, we know that there are countable partition $\mathcal{P}=\{I_{i}\}_{i=1}^{\infty}$ of the unit circle $S^{1}$ into intervals $I_{i}$ and a piecewise smooth map $f_{G}: S^{1}\rightarrow S^{1}$ so that:

   (1) the map $f_{G}$ is strictly monotonic on each $I_{i}\in \mathcal{P}$ and extends to a $C^{2}$-function on $\overline{I}_{i}$. (In fact, $f_{G}|I_{k} =g_{k}|I_{k},\, g_{k}\in G$);
\smallskip

   (2) if $f_{G}(I_{k})\cap I_{j}\neq\emptyset$,  then $f_{G}(I_{k})\supset I_{j}$;
\smallskip

   (3) for all $i$, $j$,  $\bigcup_{n=0}^{\infty}f^{n}(I_{i})\supset I_{j}.$

   The map $f_{G}$ is called a Markov map for $G$. This Markov map defines an iterated function systems (IFS).
Let us recall  the definition of an iterated function systems (IFS), see \cite{MU}.

Let $(X,\rho)$ be a non-empty compact metric space,  $I$ a countable index set with at least two elements, and $$S=\{\phi_{i}:X\rightarrow X, i\in I\}$$
a collection of injective contractions from $X$ to $X$ for which there exists
$0<s<1$ such that
$$\rho(\phi_{i}(x), \phi_{i}(y))\leq s\rho(x,y),\, i\in I,\, (x,y)\in X.$$
Any such collection of contractions is called an iterated function system.

Let $I^{n}$ denote the space of words of length $n$, $I^{\infty}$ the space of infinite sequences of symbols in $I$. Let $I^{*}=\bigcup_{n\geq1}I^{n}$ and for $\omega\in I^{n}, n\geq1,$ set
$$\phi_{\omega}=\phi_{\omega_{1}}\circ\phi_{\omega_{2}}\circ\cdot\cdot\cdot\circ\phi_{\omega_{n}}.$$
If $\omega\in I^{*}\cup I^{\infty}$ and $n\geq 1$ does not exceed the length of $\omega$, we denote by $\omega|_{n}$ the word $\omega_{1}\omega_{2}\cdot\cdot\cdot\omega_{n}.$
For $\omega\in I^{\infty}$, the set
$$\pi(\omega)=\bigcap_{n=1}^{\infty}\phi_{\omega|_{n}}(X)$$
is a singleton and therefore we can define a map $\pi:I^{\infty}\rightarrow X.$

The set
$$J=\pi(I^{\infty})=\bigcup_{\omega\in I^{\infty}}\bigcap_{n=1}^{\infty}\phi_{\omega|_{n}}(X)$$ is called the limit set associated to the system
$$S=\{\phi_{i}:X\rightarrow X, i\in I\}.$$

Let $\rho:I^{\infty}\rightarrow I^{\infty}$ be the left shift map on $I^{\infty}$, that is $\rho(\omega)=\omega_{2}\omega_{3}\cdot\cdot\cdot$. Since $\phi_{i}(\pi(\omega))=\pi(i\omega)$ for every $i\in I$, and we  get
$$\pi(\omega)=\phi_{\omega_{1}}(\pi(\rho(\omega)))$$
and
$$J=\bigcup_{i\in I}\phi_{i}(J).$$

For every $\sigma\geq 0$, we define
 $$\psi(\sigma)=\sum_{i\in I}\parallel\phi'_{i}\parallel^{\sigma}\leq\infty,$$
 where the norm $\parallel\cdot\parallel$ is the supremum norm taken over $X.$
 For $n\geq 1,$ let
 $$\psi_{n}(\sigma)=\sum_{\omega\in I^{n}}\parallel\phi'_{\omega}\parallel^{\sigma}.$$
 By \cite{MU}, we know that
 $$\psi_{n}(\sigma)<\infty\Leftrightarrow \psi(\sigma)=\psi_{1}(\sigma) <\infty.$$
 Let $\theta=\inf\{\sigma: \psi(\sigma)<\infty\}.$
 For $n\geq 1,$ the function $\log(\psi_{n})$ is convex on $(\theta, +\infty)$ and for these values of $\sigma,$
 $$P(\sigma)=\lim_{n\rightarrow\infty}\displaystyle\frac{1}{n}\log \psi_{n}(\sigma)$$
 always exists and is finite if and only if $\psi(\sigma)<\infty:$
 the function $P$ is called the topological pressure function.
By (\cite{MU}, Lemma 3.2), we know that  $P(\sigma)$ is strictly decreasing in the variable $\sigma$ on the interval $(\theta, +\infty).$ The iterated function system is regular if and only if  $P(\theta)=\infty,$ which is equivalent to $\psi(\theta)=\infty$.

 Let $G$ be a finitely generated Fuchsian group of first kind with a set of generators containing finitely many parabolic elements.
 Let $f_{G}$ be the associated Markov map. By Bowen and Series' work \cite{BS}, we know that there exists a subset $K\subset \partial\Delta$ which is the union of countable open intervals $\cup I_{i}\subset\partial\Delta$ such that the first return map $f_{K}: K\rightarrow K$
, $f_{K}(x)=f_{G}^{m(x)}(x)$, $m(x)=\inf\{m: f_{G}^{m}(x)\in K\}$ induced by $f_{G}$ satisfies an additional expanding condition:
 there exists an integer $N>0$ and a constant $\beta>1$ such that $(f_{K}^{N})'(x)\geq\beta$ for all $x\in K.$
We will use the intervals in $K$ as the index  and denote the index set by $I$. Then we get an IFS by the map $f_{K}$ as
$$S=\{\phi_{i}:\phi_{i}=f_{K}^{-1}|_{i}, \,i\in I\}.$$
These results remain valid for finitely generated second-kind Fuchsian groups as was proven by Anderson and Rocha \cite{AR} in the case of the set of generators containing no parabolic elements, but Bowen and Series result go through in this later case.

We can now prove the theorem.
\begin{proof}
Let $(\mu_{t})$ be  a  family of $G$-compatible Beltrami coefficients which  is analytic in $t\in\Delta$. By the self-similarity of the limit sets of quasi-Fuchsian groups $G_{t}$, in order  to study the dimensions of the limit set of the group $G_{t}$, it is enough to study the dimensions of the images of $K$ under quasiconformal map $f_{\mu_{t}}$. Conjugating by $f_{\mu_{t}}$, we get an IFS $S_{t}$ induced by  the IFS $S$ as $$S_{t}=\{\phi_{i}^{t}:  \phi^{t}_{i}=f_{\mu_{t}}\circ \phi_{i}\circ f^{-1}_{\mu_{t}},\, i\in I,\,t\in \Delta\}.$$

Let  $$\psi^{t}_{n}(\sigma)=\sum_{\omega\in I_{n}^{t}}\parallel(\phi_{\omega}^{t})'\parallel^{\sigma},~\, n\geq1$$
and  $P$ be the topological pressure function as follows
 $$P(t, \sigma)=\lim_{n\rightarrow\infty}\displaystyle\frac{1}{n}\log \psi_{n}^{t}(\sigma), t\in \Delta,  \sigma\in(\theta_{t}, +\infty),$$
 where $\theta_{t}=\inf_{\sigma}\{\sigma:\psi^{t}_{1}(\sigma)<\infty\}$.

If the set of generators of $G_{t}$ contains no parabolic elements, the index set $I$ is  finite. Thus $\theta_{t}=-\infty$  and the system is regular. When the set of generators of $G_{t}$ contains some parabolic elements, we need the following

   \begin{lemma}
   For any $t\in\Delta,$ $\theta_{t}=\frac{1}{2}$ and the IFS $S_{t}$  is regular.
   \end{lemma}

\begin{proof}
For fixed $t\in \Delta,$  we need to show that
$$\psi^{t}_{1}(\sigma)=\sum_{i\in I}\parallel(\phi^{t}_{i})'\parallel^{\sigma} <\infty, \, \sigma>\frac{1}{2},$$ and
$$\psi^{t}_{1}(\frac{1}{2})=\sum_{i\in I}\parallel(\phi^{t}_{i})'\parallel^{\frac{1}{2}} =\infty.$$

Without loss of generality, we may suppose that the generators $\{\gamma_{1},\cdot\cdot\cdot,\gamma_{m}, g\}$ of $G$ contains only one parabolic element $g$.
Now we divide $I$ into two parts, $\mathcal{I}_{h}$ and $\mathcal{I}_{p}$, where
$$\mathcal{I}_{h}=\{i\in I:\phi_{i} ~~~\text{is hyperbolic}\}$$
and
$$\mathcal{I}_{p}=I\setminus \mathcal{I}_{h}=\{i\in I:\phi_{i} ~~~\text{is parabolic}.\}$$
Then, we have
$$\psi^{t}_{1}(\sigma)=\sum_{i\in I}\parallel(\phi^{t}_{i})'\parallel^{\sigma}=\sum_{i\in \mathcal{I}_{h}}\parallel(\phi^{t}_{i})'\parallel^{\sigma}+\sum_{i\in \mathcal{I}_{p}}\parallel(\phi^{t}_{i})'\parallel^{\sigma}.$$

By the property of The Markov map $f_{G}$ and the definition of $f_{K},$ the index set $\mathcal{I}_{h}$ is a finite set. Hence
$$\sum_{i\in \mathcal{I}_{h}}\parallel(\phi^{t}_{i})'\parallel^{\sigma}<\infty.$$
 Since $\infty$ is an ordinary point of $G_{t},$ by \cite{Be} we know
 $$\sum_{i\in \mathcal{I}_{p}}\parallel(\phi^{t})'_{i}\parallel^{\sigma} \asymp\sum\frac{1}{n^{2\sigma}},$$
 where $A\asymp B$ means $A/C<B<CB$ for some implicit constant $C,$ the constant $C$ depends only on the number of hyperbolic generators of $G_{t}$ and the complex dilatation of $f_{\mu_{t}}.$ The lemma follows.
\end{proof}

End of the proof:
R.D. Mauldin and M. Urbanski\cite{MU} showed that for a regular system, the dimension of the limit set is the unique zero of the function $\sigma\mapsto P(t, \sigma).$
To finish the proof of the theorem, it only remains to prove that the zero varies real-analytic with respect to $t$. This follows from the classical thermodynamic formalism (a generalization of the Perron-Frobenius theorem, see \cite{Bo, Ru} ): $\exp{P(t,\sigma)}$ is an isolated eigenvalue of an transfer operator. The theorem follows from the implicit function theorem applied to $(t,\sigma)\mapsto \exp(P(t, \sigma)).$

%
  \end{proof}

\section{Proof of Theorem \ref{main3}}

Let  $\mathbb{H}$ be the upper half plane $\{z: Im(z)>0\}$, $\mathcal{D}^{*}_{1}$ the closed disk with diameter $[0, 2]$ and $\mathcal{D}^{*}_{n}$, $n\geq 1$,  the closed disk with diameter $[2^{n-1}, 2^{n}].$
We consider the domain $$\Omega=\mathbb{H}\setminus((\cup_{n\geq1}\mathcal{D}^{*}_{n})
\cup(\cup_{n\geq 1}(-\mathcal{D}^{*}_{n}))).$$

Let $\phi$ be the conformal mapping from $\Omega$ onto $\mathbb{H}$ fixing $0$, $1$ and $\infty.$ We put $z_{0}=0$, and $z_{n}=\phi(2^{n}), \, n\geq1$ and  $z_{n}=\phi(-2^{n}), \, n\leq-1.$ Let $\sigma_{n}$ be the reflection in $\mathcal{D}^{*}_{n}$ and $\tau(z)=-\bar{z}$.
By Rubel and Ryff's construction \cite{RR} of covering group of Riemann surface
$S=\mathbb{C}\setminus \{z_{n}\},$ the Fuchsian group $\Gamma$ generated by $\{\tau\circ\sigma_{n}\}^{\infty}_{n=1}$ uniformizes the surface $S$, in the sense that $S\simeq \mathbb{H}/\Gamma.$

If we return to the disk model, we then have $S\simeq \Delta/G$, where $G=\eta\circ\Gamma\circ\eta^{-1}$ and $\eta:\mathbb{H}\rightarrow\Delta$ is an isomorphism which sends $\infty$ to $i$, $1$ to $1$ and $0$ to $-i.$ Denote by $\mathcal{D}_{n}=\eta(\mathcal{D}^{*}_{n})$.
The disks $\mathcal{D}_{n}\bigcap\bar{\Delta}$ accumulate to $i$ and the diameter of $\mathcal{D}_{n}$ is comparable to $2^{-n}.$

 Let $F$ be the domain $\Delta\setminus\bigcup_{i\neq 0} \mathcal{D}_{i}$. The domain $F$ is symmetric about the $y-$axis. Let $E$ be the intersection of the closure of $F$ with the unit circle $\partial\Delta$. By the construction of  $G$ we know that $E$ contains countably many points.

 As in \cite{BS}, we denote by $\mathcal{N}$ the set of images of the sides of $F$ under $G.$ Let us also denote by $I_{n}$, $n\in \mathbb{Z}\setminus\{0\}$,
 the intersection of $\mathcal{D}_{n}$ with $\partial\Delta.$ For each point $e\in E$, we consider the set of all the elements of $\mathcal{N}$ passing through $e$ and not being a side of $F$, denote it by $\mathcal{N}_{e}$. We denote by $\mathcal{N}_{E}$ the set that contains all the elements in $\mathcal{N}$  meeting $\partial\Delta$ with only one endpoint in $E$.
 For each $n\in \mathbb{Z}$, the intervals formed by the intersection of the elements of $\mathcal{N}_{E}$ with $\partial\Delta$ then form a partition of each interval $I_{n}$. Let $e_{n-1}$, $e_{n}$(in anti-clockwise order on $\partial\Delta$) be the endpoints of $I_{n}$. For $k\leq-1$, we denote $I_{n,k}$  as the subinterval of $I_{n}$ with endpoints just as the $|k|-$th and $(|k|+1)$-th points in clockwise order of the set of the intersection of elements of
  $\mathcal{N}_{e_{n-1}}$ with $\partial\Delta.$ Similarly for $k\geq1$, we denote $I_{n,k}$  as the subinterval of $I_{n}$ with endpoints just as the $k-$th and $(k+1)$-th points in anti-clockwise order of the set of the intersection of
  elements of $\mathcal{N}_{e_{n}}$ with $\partial\Delta.$ In this case $I_{n,0}$ is just the subinterval of $I_{n}$  with endpoints just as the leftmost point in anticlockwise order of the set of the intersection of elements of
  $\mathcal{N}_{e_{n-1}}$ with $\partial\Delta$ and the rightmost point in the anticlockwise order of the set of the intersection of elements of $\mathcal{N}_{e_{n}}$ with $\partial\Delta.$   Hence we have $I_{n}=\bigcup_{k\in \mathbb{Z}}I_{n,k}.$
The set $K$ in \cite{BS}is just $\bigcup_{n\in \mathbb{Z}\setminus\{0\}}I_{n,0}.$

On each interval $I_{n}$, $n\in \mathbb{Z}\setminus\{0\}$, the Markov map $f$ is equal to $s_{n}\circ s_{0}$, where $s_{0}$ is the reflection across the $y-$axis and $s_{n}$ is the reflection across $\mathcal{D}_{n}.$ Then the induced map is equal to $f$ on $I_{n, 0}$ and to $f^{n_{k}}$ on $I_{n,k}$ where $n_{k}$ is the first integer such that $f^{n_{k}}(I_{n,k})\subset K.$

Let $I=\{(n,k), n\in \mathbb{Z}\setminus\{0\}, k\in \mathbb{Z}\}.$
For $i\in I,$ let us put $\phi_{i}=\psi_{i}^{-1}$, where  $\psi_{i}$ stand for $f|_{I_{i}}$. Then we get an IFS $S$ as
$$S=\{\phi_{i}: i\in I\}.$$

Let $(\mu_{t})$ be a family of $G-$ compatible Beltrami coefficients which is analytic in $t\in\Delta$ and $G_{t}$ the deformation group of $G$ under the quasiconformal mapping $f_{\mu_{t}}.$ Conjugating by $f_{\mu_{t}}$, we get an IFS $S_{t}$ induced by  the IFS $S$ as $$S_{t}=\{\phi_{i}^{t}:  \phi^{t}_{i}=f_{\mu_{t}}\circ \phi_{i}\circ f^{-1}_{\mu_{t}},\, i\in I,\,t\in \Delta\}.$$

 Let  $$\psi^{t}(\sigma)=\sum_{i\in I}\parallel(\phi_{i}^{t})'\parallel^{\sigma}.$$
 In order to show the Fuchsian group $G$ has Ruelle's property, we need to show the following, where the terminology comes from \cite{MU},
 \begin{lemma}
 For any $t\in \Delta,$ the IFS $S_{t}$ is regular with the $\theta$ number equal to $\frac{1}{2}.$
 \end{lemma}
 For fixed $t\in \Delta,$  we need to show that
$$\psi^{t}_{1}(\sigma)=\sum_{i\in I}\parallel(\phi^{t}_{i})'\parallel^{\sigma} <\infty, \, \sigma>\frac{1}{2},$$ and
$$\psi^{t}_{1}(\frac{1}{2})=\sum_{i\in I}\parallel(\phi^{t}_{i})'\parallel^{\frac{1}{2}} =\infty.$$

Since two connective intervals $I_{n},$ $I_{n+1}$ have comparable length, the same will be true for $f_{\mu_{t}}(I_{n}),$ $f_{\mu_{t}}(I_{n+1}),$ and thus, by bounded distortion, the above quantization
$\parallel(\phi^{t}_{n,0})'\parallel$  will be bounded by $2^{-n\alpha},$
where $\alpha$ is the H\"{o}lder exponent of the quasiconformal map $f_{\mu_{t}},$ and the bounded below by $2^{-n\beta}$, where $\beta$  is the H\"{o}lder exponent of $f^{-1}_{\mu_{t}}.$

In order to estimate $\parallel(\phi^{t}_{n,0})'\parallel$, we notice that $f_{\mu_{t}}$ conjugates the parabolic elements of $G$ to parabolic elements of $G_{t}$, from which one can get that
$$\parallel(\phi^{t}_{n,k})'\parallel\asymp\parallel
(\phi^{t}_{n,0})'\parallel\displaystyle\frac{1}{k^{2}}\leq C\displaystyle\frac{2^{-n\alpha}}{k^{2}}$$
and also
$$\parallel(\phi^{t}_{n,k})'\parallel\geq c\displaystyle\frac{2^{-n\beta}}{k^{2}},$$
where $C$ and $c$ are some constants which do not depend on the parameter $t.$
This proves the lemma.\qed

 As in the proof of Theorem \ref{main3}, R.D. Mauldin and M. Urbanski\cite{MU} showed that for a regular system, the dimension of the limit set is the unique zero of the function $\sigma\mapsto P(t, \sigma).$
 By the classical thermodynamic formalism (a generalization of the Perron-Frobenius theorem, see \cite{Ru} ) we know $\exp{P(t,\sigma)}$ is an isolated eigenvalue of a transfer operator. The theorem follows from the implicit function theorem applied to $(t,\sigma)\mapsto \exp(P(t, \sigma)).$\qed

\end{document}